\documentclass{gtpart}

\usepackage{amssymb}
\usepackage{amsmath}
\usepackage{amsxtra}

\newcommand{\En}{\mathop{\rm En}}
\newcommand{\Con}{\mathop{\rm Con}}
\newcommand{\Cut}{\mathop{\rm Cut}}

\newcommand{\Comp}{\mathop{\rm Comp~}}
\newcommand{\Ls}{\mathop{\rm Ls}}
\newcommand{\IB}{\mathop{\rm IB}}
\newcommand{\Int}{\mathop{\rm Int}}

\title{On the dynamics of subcontinua of a tree}

\author{Mykola Matviichuk}
\givenname{Mykola} \surname{Matviichuk}
\address{{National University of Kyiv, \newline Academician Glushkov ave., 2, build.7, 03127, Kyiv, Ukraine, and}\newline
        {Institute of Mathematics of NASU, \newline  Tereschenkivs'ka 3, 01601 Kyiv,
Ukraine}} \email{mykola.matviichuk@gmail.com}


\volumenumber{} \issuenumber{} \publicationyear{} \papernumber{}
\startpage{}
\endpage{}
\doi{} \MR{} \Zbl{} \received{} \revised{} \accepted{}
\published{} \publishedonline{} \proposed{} \seconded{}
\corresponding{} \editor{} \version{}

\newtheorem{theorem}{Theorem}
\newtheorem{lemma}{Lemma}
\newtheorem{remark}{Remark}

\begin{document}

\begin{abstract}
Given a tree map $f:T\to T$, we study the dynamics of subcontinua
of $T$ under action of $f$. In particular, we prove that a
subcontinuum of $T$ is either asymptotically periodic or
asymptotically degenerate. As an application of this result, we
show that zero topological entropy of the system $(T,f)$ implies
zero topological entropy of its functional envelope (endowed with
the Hausdorff metric).
\end{abstract}

\maketitle

\section{Introduction}

By a \emph{(topological) dynamical system} we mean a pair $(X,f)$
where $X$ is a compact metrizable topological space and $f:X\to X$
is a map, i.e. continuous function. Recall that a \emph{continuum}
is a nonempty compact connected metric space. Given a dynamical
system $(X,f)$, one can in a natural way extend $f$ to a map
$\mathcal{F}$ on the hyperspace $\Con(X)$ of all subcontinua of
$X$. We call the system $(\Con(X),\mathcal{F})$ a \emph{connected
envelope} (where $\Con(X)$ is endowed with the topology induced by
the Hausdorff metric). The natural question arises here: what is
the connection between dynamical properties of the base map $f$
and its extension $\mathcal{F}$. For papers related to this topic,
see \cite{AIM-L}, \cite{Banks}, \cite{FRSh}, \cite{KwOp}.

In the present paper we deal with the case when underlying phase
space is a tree. At the end of the paper we will prove
(Theorem~\ref{Tm:EntrCon}) the equality of topological entropies
of a dynamical system on a tree and its connected envelope. As a
consequence, we will get a nice result concerning a system which a
dynamical system on a tree induces on the hyperspace of all maps
on this tree endowed with the Hausdorff metric; following
\cite{AKoS} we call it a \emph{functional envelope}. Namely, we
prove (Theorem~\ref{Tm:EntrFun}) that if a system on a tree has
zero topological entropy, then so does its functional envelope
(cf. with result due to Glasner and Weiss \cite{GW} who proved
that zero entropy of any topological dynamical system implies zero
entropy of the system induced on the space of all probability
Borel measures on the phase space). For the case of interval both
these results were done in \cite{Matv1}.

In order to prove the mentioned results we study the dynamics of a
subcontinuum of a tree under action of a tree map. First, in
Section~\ref{WihtPer}, we consider the situation when the
subcontinuum contains a periodic point of the map. In \cite{Fed}
it was proved that if a subinterval of an interval contains a
periodic point of an interval map, then it is asymptotically
periodic with respect to this map. We prove (see
Theorem~\ref{Th:PerToAsPer}) the generalization of this result for
tree maps, i.e. we prove that each subcontinuum of a tree
containing a periodic point of a tree map is asymptotically
periodic with respect to the map. Unfortunately, our method does
not provide a good estimate of period of the asymptotically
periodic set. For the case of interval such an estimate is known;
namely, the period of the set is a divisor of doubled period of
each periodic point it contains \cite{Fed}.

Next, in Section~\ref{WithoutPer}, we consider in some sense the
opposite situation, when only the endpoints of a tree are
permitted to be periodic. Recall that, by the fixed point
property, it must have at least one of them. It turns out that in
this setting there is a unique attracting fixed point which
attracts everything which does not eventually glue to a periodic
orbit (see Lemmas~\ref{Lm:UnAttr} and \ref{Lm:PointAsPer}). As a
consequence, we get that any subcontinuum of the tree converges to
the attracting fixed point, provided that it does not glue to a
periodic orbit; and if it does, then, by previous results, it is
asymptotically periodic (see Theorem~\ref{Th:NoPerToAsPer} and the
proof).

Finally, in Section~\ref{MainAndEntropy}, we prove that any
subcontinuum of a tree when it is iterated under a tree map is
either asymptotically periodic or asymptotically degenerate, or
both (see Theorem~\ref{Tm:Dicho}). For interval maps such a
characterization was known (see for instance \cite{Matv1}) and for
transitive graph maps similar result was recently proved in
\cite{KwOp}. Still for general graph maps the situation is
unclear. We finish the paper with the above-mentioned result that
zero entropy of a tree dynamical system implies zero entropy of
its functional envelope. We remark that this phenomenon is
essentially due to dimension one. There are quite simple examples
of zero entropy maps on the square for which the functional
envelope has infinite entropy (e.g. $f(x,y)=(x,y^2)$,
$(x,y)\in[0,1]^2$ works). So, the following open question seems to
be quite natural here.

\emph{Question.} Does Theorem~\ref{Tm:EntrFun} remain true for a)
graphs with loops, b) dendrites?

\section{The dynamics of a subcontinuum of a tree containing a periodic
point.}\label{WihtPer}

First, let us recall some definitions and fix notations. By an
\emph{interval} we mean any space homeomorphic to $[0,1]\subset
\mathbb{R}$. A \emph{tree} is a uniquely arcwise connected space
that is either a point or a union of finitely many intervals.
Remark that any tree is a continuum. Any continuous function from
a tree into itself is called a \emph{tree map}. If $T$ is a tree
and $x\in T$, we define the \emph{valence} of $x$ to be the number
of connected components of $T\setminus\{x\}$. Each point of
valence one will be called an \emph{endpoint} of $T$ and the set
of such points will be denoted by $\En(T)$. A point of valence
greater than one will be called a \emph{cut-point} and the set of
cut-points of $T$ will be denoted by $\Cut(T)$. A point of valence
different from two will be called a \emph{vertex} of $T$, and the
set of vertices of $T$ will be denoted by $V(T)$. The closure of
each connected component of $T\setminus V(T)$ will be called an
\emph{edge} of $T$.

If $(X,f)$ is a dynamical system and $x\in X$ then the
\emph{$\omega$-limit set} of $x$ under $f$ is the set
$\omega_f(x)$ of all limit points of the trajectory
$x,f(x),f^2(x),\dots$ regarding it as a sequence. Given a subset
$A$ of a topological space, we denote by $\overline{A}$, $\Int(A)$
and $\partial A$ the closure, the interior and the boundary of
$A$, respectively. Moreover, for $x\in X$ we will denote by
$\Comp(A,x)$ the (connected) component of $A$ containing $x$ if
$x\in A$, and the singleton $\{x\}$ if $x\notin  A$. For a finite
set $B$ we will denote its cardinality by $|B|$.

Let us summarize some simple topological facts we will need. Let
$T$ be a tree, $M$ be a subcontinuum of $T$ and $A$, $A_n, n\ge 0$
be connected subsets of $T$. Then the following holds.
\begin{itemize}
    \item $M$ is a tree. Also the factor space $T/_M$ (i.e. we just
    identify all points within $M$) is a tree.
    \item The set $\partial A$  is finite.
    \item Each point in $\overline{A}\setminus A$ is an endpoint of
    $\overline{A}$.
    \item The set $\overline{\Comp(T\setminus
    A, x)}\cap\overline{A}$ is a singleton for each $x\in T$.
    \item If $A_n\cap A_{n+1}\not=\emptyset$ for each $n\ge0$, then
    $\cup_{n=0}^{\infty} A_n$ is again a connected set.
    \item The set  $\cap_{n=0}^{\infty} A_n$ is either connected
    or empty.
\end{itemize}

Given a tree $T$, a sequence $\{x_n\}_{n=0}^{\infty}\subseteq T$
is said to be \emph{consistent with $x\in T$} if
$x_m\in\Comp(T\setminus\{x_n\},x)$ whenever $m>n\ge 0$. Of course,
a sequence  $\{x_n\}_{n=0}^{\infty}$ which is consistent with some
$x$ does not need to be convergent; consider the example
$T=[-1,1]$, $x=0$ and $x_n=(-1)^n\cdot\dfrac{n+1}{2n}$, $n\ge 1$.
However, as it is in the example, one can always split the
sequence into a finite number of convergent (and, in some sense,
monotone) subsequences.

\begin{lemma}\label{lm:consistent seq}
Let $T$ be a tree, $x\in T$ and $\{x_n\}_{n=0}^{\infty}\subseteq
T$ be a sequence consistent with $x$. Then there is a finite
partition of the set of nonnegative integers into the sets
$L_1,L_2,\dots,L_k$ such that $[x,x_m]\subseteq [x,x_n]$ whenever
$m>n$ and $m,n\in L_i$ for some $1\le i\le k$. In particular, each
subsequence $\{x_n\}_{n\in L_i}$, $1\le i\le k$ is convergent.
\end{lemma}

\begin{proof} The proof is straightforward. We just express $T$ as
the union $\cup_{i=1}^{k}[x,y_i]$ where $y_1,y_2,\dots,y_k$ is an
enumeration of all endpoints of $T$, and then define each $L_i$ to
be the set of those indices $n$ for which $x_n\in[x,y_i]$ but
$x_n\notin[x,y_j]$ for any $j<i$.
\end{proof}

Given a metric space $X$, we denote by $\Con(X)$ the space of all
subcontinua of $X$ endowed with the following topology. For a
sequence $\{A_n\}_{n=0}^{\infty}\subset\Con(X)$ we define:
\begin{eqnarray*}
\liminf A_n &=\{&x\in X : \text{if}~ U ~\text{is an open subset
of}~ X~ \text{with}~ U\ni x, \\&& \text{then}~ U\cap A_i
\not=\emptyset ~\text{for all but finitely many}~ n\};
\end{eqnarray*}
\begin{eqnarray*}
\limsup A_n &=\{&x\in X : \text{if}~ U ~\text{is an open subset
of}~ X~ \text{with}~ U\ni x, \\&&\text{then}~ U\cap A_i
\not=\emptyset ~\text{for infinitely many}~ n\}.
\end{eqnarray*}

In fact, $\liminf A_n, \limsup A_n\in\Con(X)$ and $\liminf
A_n\subseteq\limsup A_n$. If~$\liminf A_n$ $ = A = \limsup A_n$,
then we say that $\{A_n\}_{n=0}^{\infty}$ \emph{converges to} $A$
\emph{as} $n\to\infty$, written $A_n\to A$, $n\to\infty$. It is
well known that this convergence defines a topology on $\Con(X)$,
and $\Con(X)$ endowed with this topology is a compact metrizable
topological space. In fact, this topology is given by the
Hausdorff metric, which we will define later when we need it
explicitly.

The following easy lemma shows that, given a tree $T$ and
$M\in\Con(T)$, convergence in the space $\Con(T)$ is given by
convergence in the spaces $\Con(M)$ and $\Con(T/_M)$. Denote by
$\pi_M$ the canonical projection $T\to T/_M$.

\begin{lemma}\label{lm:1} Let $T$ be a tree and $A_n, n\ge1, M\in\Con(T)$. Suppose
that $M\cap A_n\not=\emptyset$ for each $n$ and both the sequences
$\{M\cap A_n\}_{n=0}^{\infty}\subseteq\Con(M)$ and
$\{\pi_M(A_n)\}_{n=0}^{\infty}\subseteq\Con(T/_M)$ converge in the
corresponding spaces. Then the sequence $\{A_n\}_{n=0}^{\infty}$
converges in $\Con(T)$.
\end{lemma}

\begin{proof}
If $x\in T\setminus M$, then one can take an open set $U\ni x$
such that $U\cap M=\emptyset$. So, each $x\in T\setminus M$
belongs to $\liminf A_n$ (resp. $\limsup A_n$) iff $x$ belongs to
$\liminf \pi_M(A_n)$ (resp. $\limsup \pi_M(A_n)$). Next, we are
going to prove $M\cap\liminf A_n=\liminf (M\cap A_n)$ and
$M\cap\limsup A_n=\limsup (M\cap A_n)$. To achieve this, it
suffices to show that, given $x\in M$ and a connected open subset
$U$ of $T$ with $U\ni x$, if $U$ intersects $A_n$ then it
intersects $M\cap A_n$, for each $n$. Let $x$ and $U$ be as above
and assume that $U\cap A_n\not=\emptyset$. We take $y\in U\cap
A_n$, $z\in M\cap A_n$ and $u\in [z,y]\subseteq A_n$ such that
$[z,u]=[z,y]\cap M$. Then $u\in M\cap A_n$, and so it is enough to
show $u\in U$. To this end, observe that
$u\in\overline{\Comp(T\setminus M, y)}$, and hence
$\overline{\Comp(T\setminus M, y)}\cap M=\{u\}$. Let $v\in
[x,y]\subseteq U$ be such that $[x,v]=[x,y]\cap M$. Then $v\in
M\cap U$ and $v\in\overline{\Comp(T\setminus M, y)}$. Thus
$\{v\}=\overline{\Comp(T\setminus M, y)}\cap M=\{u\}$ which leads
to $u=v\in U$.

To sum it up, we have proved that $(T\setminus M)\cap \liminf
A_n=(\liminf \pi_M(A_n))\setminus \pi_M(M)$ and $M\cap\liminf
A_n=\liminf (M\cap A_n)$, and also the same with $\liminf$
replaced by $\limsup$. Therefore, $\liminf A_n=\limsup
A_n=\left(A'\setminus \pi_M(M)\right)\cup A''$, where $A'$ and
$A''$ denote the limits of the sequences
$\{\pi_M(A_n)\}_{n=0}^{\infty}$ and $\{M\cap A_n\}_{n=0}^{\infty}$
respectively.
\end{proof}

Given a dynamical system $(X,f)$, a set $M\subseteq X$ is called
\emph{invariant} (resp. \emph{strongly invariant}) if
$f(M)\subseteq M$ (resp. $f(M)=M$). For a subset $A\subseteq X$,
we denote by $\Ls(f,A)$ the set-theoretical limit superior of the
sequence $\{f^n(A)\}_{n=0}^{\infty}$, i.e.
$\Ls(f,A)=\cap_{m=0}^{\infty}\cup_{n= m}^{\infty}f^n(A)$.

\begin{lemma}\label{lm:Ls}
Let $f:T\to T$ be a tree map and $A\in\Con(T)$ contains a fixed
point $x$ of $f$. Then $\Ls(f,A)$ is strongly invariant connected
subset of $T$ containing $x$.
\end{lemma}

\begin{proof}
Let $\Delta =\Ls(f,A)= \cap_{m=0}^{\infty} \Delta_m$, where
$\Delta_m=\cup_{n=m}^{\infty} f^n(A)$ for any $m\ge 0$. First,
$\Delta$ is a connected set containing $x$, because each $f^n(A)$
is. Next, since $f(\Delta_m)=\Delta_{m+1}$ for each $m\ge0$ and
$\Delta_m$ decreases on $m$, the set $\Delta$ is invariant as
intersection of a family of invariant sets. On the other hand, fix
any $x\in\Delta$ and let us show that $x=f(y)$ for some
$y\in\Delta$. Whatever the $m\ge0$, from
$x\in\Delta_{m+1}=f(\Delta_m)$ we get $x=f(y_m)$ for some
$y_m\in\Delta_m$. Consider the sequence $\{y_m\}_{m=0}^{\infty}$.
Let $m_0=0$. If $y_m\notin\Comp(T\setminus\{y_{m_0}\},x)$ for
infinitely many $m\ge m_0$, then $y_{m_0}\in[x,y_m]\subseteq
\Delta_m$ for infinitely many $m\ge m_0$, and so
$y_{m_0}\in\Delta$. Otherwise, there is $m_1>m_0$ such that
$y_m\in\Comp(T\setminus\{y_{m_0}\},x)$ for all $m\ge m_1$. On the
next step, if $y_m\notin\Comp(T\setminus\{y_{m_1}\},x)$ for
infinitely many $m\ge m_1$, then $y_{m_1}\in[x,y_m]\subseteq
\Delta_m$ for infinitely many $m\ge m_1$, and so
$y_{m_1}\in\Delta$. Otherwise, there is $m_2>m_1$ such that
$y_m\in\Comp(T\setminus\{y_{m_1}\},x)$ for all $m\ge m_2$.
Repeating this procedure, we either get that $x=f(y_{m_r})$,
$y_{m_r}\in\Delta$ for some $r\ge0$ or get the subsequence
$\{y_{m_r}\}_{r=1}^{\infty}$ which is consistent with $x$ (see the
definition of consistent sequence before Lemma~\ref{lm:consistent
seq}) and such that $x=f(y_{m_r})$, $y_{m_r}\in\Delta_{m_r}$ for
each $r\ge0$. In the former case we get exactly what we need to
complete the proof of strong invariance of $\Delta$. In the latter
case, applying Lemma~\ref{lm:consistent seq}, we get a convergent
subsequence $\{y_{m_r}\}_{r\in L}$ (here $L$ is an infinite subset
of the set of nonnegative integers) such that
$[x,y_{m_s}]\subseteq [x,y_{m_r}]$ whenever $s>r$ and $s,r\in L$.
Therefore, if $y$ denotes the limit of $\{y_{m_r}\}_{r\in L}$,
then $y\in\cap_{r\in L}[x,y_{m_r}]\subseteq \cap_{r\in
L}\Delta_{m_r}=\Delta$ and, by continuity, $x=f(y)$. So, we have
showed that the set $\Delta$ is strongly invariant.
\end{proof}

Given a dynamical system $(X,f)$ and a nonempty, closed and
invariant set $M\subseteq X$, one can consider a \emph{subsystem}
$(M,f|_M)$, where $f|_M$ is the restriction of $f$ to $M$. In the
same setting, one can define a \emph{factor-system} $(X,f)/_M :=
(X/_M,f/_M)$, where $X/_M$ is the factor space and $f/_M:X/_M\to
X/_M$ is given by $f/_M=\pi_M\circ f \circ\pi_M^{-1}$ where
$\pi_M:X\to X/_M$ is the canonical projection.

Let $f:T\to T$ be a tree map. A continuum $A\in\Con(T)$ is called
\emph{asymptotically periodic} under $f$ if the sequence
$f^{pn}(A), n\ge0$ converges for some $p\ge1$.

\begin{theorem}\label{Th:PerToAsPer}
Let $f:T\to T$ be a tree map and $A\in\Con(T)$ contains a periodic
point of $f$. Then $A$ is asymptotically periodic under $f$.
\end{theorem}

\begin{proof}
Let $x\in A$ be a periodic point. In the sequel we will freely
replace $f$ with $f^k$ and $A$ with $f^m(A)$ for some positive
integers $k,m$, because it is enough to prove that the sequence
$f^{pkn+m}(A), n\ge 0$ converges for some $p\ge1$. Thus, at first,
it is convenient to assume that $x$ is a just fixed point.

Let $\Delta=\Ls(f,A)$. By Lemma~\ref{lm:Ls}, the set $\Delta$ is a
connected strongly invariant set containing $x$. If it happens
that $\Delta=\{x\}$ then we are done, because we easily get
$f^n(A)\to \{x\}, n\to\infty$. Otherwise, we express
$\overline{\Delta}$ as the union $\cup_{i=1}^k[x,x_i]$ where
$k=|\En(\overline{\Delta})\setminus\{x\}|$ and
$\{x_1,x_2,\dots,x_k\}$ is an enumeration of all endpoints of
$\overline{\Delta}$ but possibly $x$ (if $x\in
\En(\overline{\Delta})$). Here some of $x_i$'s belong to $\Delta$,
while the others belong to $\overline{\Delta}\setminus\Delta$.
Next, passing to subsystems and factor-systems, we will decrease
the number of $x_i$'s.

First, we consider the case when all $x_i$'s belong to $\Delta$,
i.e. $\Delta$ is closed. Since $f(\Delta)=\Delta$ and
$\Delta=\cup_{i=1}^k[x,x_i]$, for each $1\le i\le k$ there is
$1\le j\le k$ such that $f[x,x_j]\supseteq[x,x_i]$. Hence, there
are $1\le i,j\le k$ such that $f^j[x,x_i]\supseteq [x,x_i]$. By
definition of $\Delta$, there is a positive integer $m$ such that
$x_i\in f^m(A)$. Replacing $A$ with $f^m(A)$ and $f$ with $f^j$,
we can assume that $x_i\in A$ and $f[x,x_i]\supseteq [x,x_i]$.
Now, we are going to prove that $f^{pn}(A)$ converges as
$n\to\infty$, for some $p\ge1$. To this end, we are going to use
Lemma~\ref{lm:1} for $M=\cup_{n\ge 0}f^n[x,x_i]$ and for the
sequence $A_{np}=f^{np}(A)$, $n\ge0$. Since clearly $M\cap A_n\to
M$, $n\to\infty$, all we need is to prove that $\pi_M(A_{pn})$
converges as $n\to\infty$, for some $p\ge1$. We remark that
$\pi_M(f^{n}(A))= g^{n}(B)$, $n\ge 0$ where $g=f/_M$ and
$B=\pi_M(A)$. Thus we consider the factor-system $(T/_M,g)$ and
continuum $B=\pi_M(A)\subseteq T/_M$ which contains the fixed
point $\pi_M(x)$ of $g$. Therefore, we have reduced the proving of
asymptotical periodicity of $A$ under $f$ to the proving of
asymptotical periodicity of $B$ under $g$. The set
$\Ls(g,B)=\pi_M(\Delta)$ is again, by Lemma~\ref{lm:Ls}, a
strongly invariant continuum containing the fixed point, but now
$\En(\Ls(g,B))\setminus \{\pi_M(x)\}$ has at most $k-1$ elements,
because $\pi_M[x,x_i]=\{\pi_M(x)\}$. By repeating this procedure
we will eventually get that $\Delta=\{x\}$, and so the proof is
complete for the case of closed $\Delta$.

Now, assume that $\Delta$ is not closed. Since $\Delta$ is
strongly invariant, $\overline{\Delta}\setminus\Delta$ is also
strongly invariant. Moreover, $\overline{\Delta}\setminus\Delta$
is finite as it is contained in the boundary of connected subset
of a tree. Replacing $f$ with
$f^{|\overline{\Delta}\setminus\Delta|!}$, we can assume that all
points of $\overline{\Delta}\setminus\Delta$ are fixed under $f$.
Let $x_i\in\overline{\Delta}\setminus\Delta$, $f(x_i)=x_i$. Let us
show that $\overline{\Delta_m}$ contains only one preimage of
$x_i$ for each $m$ which is large enough, where
$\Delta_m=\cup_{n=m}^{\infty} f^n(A)$. In order to see this, note
that $\{x\in\overline{\Delta} : f(x)=x_i\}$ is just a singleton
$\{x_i\}$, for the set $\Delta$ is invariant and each point in
$\overline{\Delta}\setminus\Delta$ is fixed. Choose $m'$ such that
$x_i\notin\Delta_{m'}$. Then
$x_i\in\overline{\Delta_{m'}}\setminus\Delta_{m'}$, in particular,
$x_i$ is endpoint of $\overline{\Delta_{m'}}$. Consider the closed
set $\{x\in\overline{\Delta_{m'}} : f(x)=x_i\}$. As we remarked
above, it intersects $\overline{\Delta}$ at exactly one point
$x_i$. Moreover, $x_i$ is isolated in $\overline{\Delta_{m'}}$,
for $x_i$ is an endpoint for both $\overline{\Delta_{m'}}$ and
$\overline{\Delta}$. Therefore, $\{x\in\overline{\Delta_{m}} :
f(x)=x_i\}=\{x_i\}$ for each $m\ge m'$ which is large enough. So,
replacing $(T, f)$ with the subsystem $(\overline{\Delta_{m}},
f|_{\overline{\Delta_{m}}})$, we can assume that $x_i$ is an
endpoint of $T$ and $f^{-1}(x_i)=\{x_i\}$.

Since $x_i\notin\Delta_{m}\supseteq f^{m}(A)$, there is a small
enough neighbourhood $[x_i,y)$ of $x_i$ such that $T\setminus
[x_i,y)\supseteq f^{m}(A)$. It follows that if a closed invariant
set contains $T\setminus [x_i,y)$, it must coincide with whole
$\overline{\Delta_{m}}=T$. Since $f^{-1}(x_i)=\{x_i\}$, we can
take a neighbourhood $[x_i,z)\subseteq [x_i,y)$ of $x_i$ such that
$T\setminus [x_i,z)\supseteq f(T\setminus [x_i,y))$. Then
$f[x_i,z]\supset[x_i,z]$, for otherwise the set $T\setminus
[x_i,z)=(T\setminus [x_i,y))\cup [x,z]$ would be proper closed
invariant subset of $T$ which contains $T\setminus [x_i,y)$.
Similarly we get $\overline{\cup_{n\ge 0}f^n[x,z]}\ni x_i$, for
otherwise the set $(T\setminus
[x_i,y))\cup\left(\overline{\cup_{n\ge 0}f^n[x,z]}\right)$ would
be proper closed invariant subset of $T$ which contains
$T\setminus [x_i,y)$. Now, we using Lemma~\ref{lm:1} pass to the
factor-system $(T/_M,f/_M)$, where $M=\overline{\cup_{n\ge
0}f^n[x,z]}$. Putting $A_n=f^n(A)$, $n\ge m$ we get $M\cap A_n\to
M$, $n\to\infty$, so we need only to show that $\pi_M(A_{pn})$
converges as $n\to\infty$, for some $p\ge1$. Since
$\pi_M(A_{n+m})=g^n(B)$ where $g=f/_m$ and $B=\pi_M(f^m(A))$, we
need only to prove that the continuum $B$, which contains the
fixed point $\pi_M(x)$, is asymptotically periodic under the tree
map $g$. We remark that $\Ls(g,B)=\pi_M(\Delta)$, and
$\overline{\Ls(g,B)}\setminus\Ls(g,B)=\pi_M(\overline{\Delta})
\setminus\pi_M(\Delta)$ consist of at most
$|\overline{\Delta}\setminus\Delta|-1$ points, for
$\pi_M(x_i)=\pi_M(x)\in\pi_M(\Delta)$. Thus step by step we reduce
the general case to the case when $\overline{\Delta}\setminus
\Delta$ is empty, i.e. $\Delta$ is closed (this case was
considered earlier).
\end{proof}

\section{The dynamics of a tree system without periodic
cut-points.}\label{WithoutPer}

Recall that, given a map $f:X\to X$, a point $x\in X$ is called an
\emph{attracting fixed point} (AFP, for short) if for any open set
$U\ni x$ there is an open set $U\supseteq V\ni x$ such that
$f(\overline{V})\subseteq V$. Let $f:T\to T$ be a tree map such
that no cut-point of $T$ is fixed under $f$. Then one can easily
see that if $x\in\En T$ such that $f(y)\in[x,y)$ for some $y$
within the edge of $T$ containing $x$, then $x$ is AFP. On the
other hand, if $x\in\En T$ is AFP, then $f(y)\in[x,y)$ for each
$y$ within the edge of $T$ containing $x$.

\begin{lemma}\label{Lm:UnAttr}
Let $f:T\to T$ be a tree map such that no cut-point of $T$ is
fixed. Then there is unique AFP of $f$ in $T$ (which is, of
course, an endpoint of $T$).
\end{lemma}

\begin{proof} \emph{Existence}. We will say that a point $y\in \Cut(T)$ \emph{moves towards} $x\in
\En(T)$ if $f(y)\in\Comp(T\setminus \{y\},x)$ (equivalent
condition is $y\notin[x,f(y)]$).

\textbf{Claim.} \emph{For each $1\le k\le|\En(T)|-1$ there is a
cut-point $y$ and an endpoint $x$ such that $y$ moves towards $x$
and $\Comp(T\setminus \{y\},x)$ contains at most $k$ endpoints of
$T$.}

For $k=|\En(T)|-1$ our claim is clear, because we can take
arbitrary $y\in\Cut(T)$ and then any endpoint $x$ from
$\Comp(T\setminus \{y\},f(y))$, so one can see that our claim
holds for the chosen $x$ and $y$. By induction, assume we have
proved the claim for some $k$ and let us prove it for $k-1$.

So, suppose that a cut-point $y$ moves towards an endpoint $x$ and
$\Comp(T\setminus \{y\},x)$ contains at most $k$ endpoints of $T$.
We take $y_1$ close enough to $x$ so that $y_1$ belongs to the
edge of $T$ containing $x$ and $[x,y_1]\subseteq[x,y)$. If $x$ is
AFP, then we are done, otherwise we get $f(y_1)\notin[x,y_1)$. The
latter is equivalent to $y_1\in[x,f(y_1))$. Let us define a
continuous map $g:[y_1,y]\to[y_1,y]$ by $g=Pr_{[y_1,y]}\circ
f|_{[y_1,y]}$, where $Pr_{[y_1,y]}$ denotes the "projection" onto
the set $[y_1,y]$, i.e. $Pr_{[y_1,y]}(z)$ is the unique point in
$\overline{\Comp(T\setminus[y_1,y],z)}\cap[y_1,y]$. By the fixed
point property, there is $y_2\in[y_1,y]$ such that $g(y_2)=y_2$.
Therefore, $y_2=Pr_{[y_1,y]}(f(y_2))$, and so
$[y_2,f(y_2)]\cap[y_1,y]=\{y_2\}$. Moreover $y_2\in[y_1,y)$,
because $y_1\in[x,y)$ and $y$ moves towards $x$. So,
$y\notin[x,y_2]$ which leads to $y_2\in\Comp(T\setminus\{y\},x)$.
This means that all the components of $T\setminus\{y_2\}$ but
$\Comp(T\setminus\{y_2\},y)$ are subsets of
$\Comp(T\setminus\{y\},x)$. On the other hand, $y_2\in[y,f(y_2)]$,
which means that $\Comp(T\setminus\{y_2\},f(y_2))$ is subset of
$\Comp(T\setminus\{y\},x)$, while it does not contain all the
endpoints of $T$ which are within $\Comp(T\setminus\{y\},x)$
(namely, it does not contain $x$). Thus,
$\Comp(T\setminus\{y_2\},f(y_2))$ contains at most $k-1$ endpoints
of $T$, and the claim follows.

In particular, if $k=1$ in the claim above, we get that there is a
cut-point $y$ moving towards an endpoint $x$ and such that
$\Comp(T\setminus\{y\},x)$ is just the semi-open interval $[x,y)$.
So, $x$ is AFP.

\emph{Uniqueness}. On the contrary, suppose that there are two
distinct AFP's $x$ and $x'$. Consider the set
$$
W=\{y\in(x,x'): y ~\text{moves towards}~ x \}
$$
By continuity, both the sets $W$ and $(x,x')\setminus W$ are
nonempty and open in $(x,x')$, which contradicts connectedness of
$(x,x')$.
\end{proof}

\begin{remark}\label{rm:attr_of_f^n}
If a tree map $f:T\to T$ is free of periodic cut-points, then for
each iterate $f^n$ the unique AFP is well defined and coincides
with that of $f$. The reason for that is the following. If $s\in
\En(T)$ is the AFP of $f$, then for each neighbourhood of the form
$[s,y)$ we have $f[s,y)\subset[s,y)$. Thus $f^n[s,y)\subset[s,y)$
for each $n\ge 0$, and so $s$ is the AFP of each $f^n$.
\end{remark}

\begin{remark}\label{rm:move_towards}
One can show, in the same way as in the proof of uniqueness above,
that each cut-point of $T$ moves towards the AFP. Thus, taking
into account Remark~\ref{rm:attr_of_f^n}, we see that if $x\in T$,
then either $f^n(x)\in\En(T)$ for some $n\ge0$ or the sequence
$\{f^n(x)\}_{n=0}^{\infty}$ is consistent with the AFP.
\end{remark}

Next, we describe the dynamics of points and subcontinua in the
system on a tree without periodic cut-points.

\begin{lemma}\label{Lm:PointAsPer} Let $f:T\to T$ be a tree map such that no
cut-point of $T$ is periodic. Let $s\in \En(T)$ be its unique AFP.
Then for each $x\in T$ either

(a) $f^n(x)$ is a periodic cut-point for some $n\ge 0$, or

(b) $f^n(x)\to s, n\to\infty$.
\end{lemma}

\begin{proof}
Let us suppose that no iterate $f^n(x),n\ge 0$ is periodic and
prove that $f^n(x)\to s, n\to\infty$. According to
Remark~\ref{rm:move_towards} after Lemma~\ref{Lm:UnAttr}, the
sequence $\{f^n(x)\}_{n=0}^{\infty}$ is consistent with $s$.
Therefore, by Lemma~\ref{lm:consistent seq}, the $\omega$-limit
set of $x$ is a finite subset of $\Cut(T)\cup\{s\}$. Once $\Omega$
is finite, it must contain a periodic point, for $\Omega$ is an
invariant set. Once $\Omega\subset \Cut(T)\cup\{s\}$, the only
periodic point it may contain is $s$. So $s\in \Omega$. Then we
immediately get $f^n(x)\to s, n\to\infty$, because $s$ is an AFP.
\end{proof}

Let $f:T\to T$ be a tree map which is free of periodic cut-point,
$s\in\En(T)$ be its unique AFP and $[s,y]$ be the edge of $T$
containing $s$. By the \emph{immediate basin of attraction} of $s$
we mean the open set
$\IB(s)=\Comp(\cup_{n=0}^{\infty}f^{-n}[s,y),s)$. It is not hard
to see that both $\IB(s)$ and $\partial\,\IB(s)$ are invariant
sets. Clearly, if $A\in\Con(T)$ is a subset of the immediate basin
of attraction of $s$, then $f^n(A)\to\{s\}, n\to\infty$. Of
course, the immediate basin of attraction of $s$ does not need to
contain all cut-point of $T$, in other words, $f^n(A)$ does not
need to converge to $\{s\}$ even if $A\subset\Cut(T)$. However, as
we will see, the only way to escape converging to $s$ is to
'cling' to some of other periodic end-points of $T$.

\begin{theorem}\label{Th:NoPerToAsPer} Let $f:T\to T$ be a tree map such that
no cut-point of $T$ is periodic. Then each $A\in\Con(T)$ is
asymptotically periodic under $f$.
\end{theorem}

\begin{proof}
Fix any $A\in\Con(T)$. Then, by Lemma~\ref{Lm:PointAsPer}, either
$f^m(A)$ contains a periodic point for some $m$, or $f^m(A)$
intersects the immediate basin of attraction $\IB(s)$ for some
$m$. In the former case $A$ is asymptotically periodic in view of
Theorem~\ref{Th:PerToAsPer}. In the latter one we consider two
subcases: $f^m(A)\subseteq \IB(s)$ and $f^m(A)\nsubseteq \IB(s)$,
but $f^m(A)\cap IB(s)\not=\emptyset$. If $f^m(A)\subseteq \IB(s)$,
then we get $f^n(A)\to\{s\}, n\to\infty$. If $f^m(A)\nsubseteq
\IB(s)$, but $f^m(A)\cap IB(s)\not=\emptyset$, then $f^m(A)$
intersects $\partial \IB(s)$. As we remarked above, $\partial
\IB(s)$ is an invariant set. Moreover, it is finite as boundary of
a connected subset of a tree. So $f^{m+k}(A)$ contains a periodic
point for some $k$ and we, using again
Theorem~\ref{Th:PerToAsPer}, deduce asymptotical periodicity of
$A$.
\end{proof}

\section{Entropy of induced systems for tree
maps}\label{MainAndEntropy}

In this section, using our previous results, we will compute the
topological entropy of connected envelope and functional envelope
of a dynamical system on a tree. Throughout the section we will
regard a tree as a metric, rather than topological, space.

First, we give the following description of the dynamics of
subcontinua of a tree (cf. Proposition in \cite{Matv1}). The proof
just mixes Theorems~\ref{Th:PerToAsPer} and \ref{Th:NoPerToAsPer}.
Given a tree map $f:T\to T$, an element $A\in \Con(T)$ is called
\emph{asymptotically degenerate} under $f$ if $diam\,f^n(A)\to 0,
n\to\infty$, where $diam$ stands for diameter of the set.

\begin{theorem}\label{Tm:Dicho}
Let $f:T\to T$ be a tree map. Then each $A\in \Con(T)$ is either
asymptotically periodic or asymptotically degenerate under $f$ (or
both).
\end{theorem}

\begin{proof}
Fix $A\in \Con(T)$. If all iterates $f^n(A)$ are pairwise
disjoint, then obviously $A$ is asymptotically degenerate. So, we
assume that $f^k(A)\cap f^m(A)\not=\emptyset$ for some $m>k\ge 0$.
Replacing $A$ with $f^k(A)$ and $f$ with $f^{m-k}$ we can assume
that $A\cap f(A)\not=\emptyset$. Then the set $\cup_{n\ge
0}f^n(A)$ is an invariant connected subset of $T$. Passing to the
subspace we can assume that $T=\overline{\cup_{n\ge 0}f^n(A)}$.
Now, if there is a periodic cut-point in $T$, then it belongs to
some $f^n(A)$, and thus by Theorem~\ref{Th:PerToAsPer} $A$ is
asymptotically periodic. On the other hand, if no cut-point of $T$
is periodic, then by Theorem~\ref{Th:NoPerToAsPer} $A$ is
asymptotically periodic, too.
\end{proof}

The notion of \emph{topological entropy} of a system on a compact
topological space was introduced by Adler, Konheim and McAndrew in
\cite{AKM} as a measure of chaotic character of a dynamical
system. In this paper we will use the Bowen-Dinaburg's definitions
of the topological entropy (see e.g. \cite{Bow}) for systems on
compact metric spaces, which agree with Adler-Konheim-McAndrew's
one for systems on topological metrizable spaces. Let $(X,\rho)$
be a compact metric space and let $f: X\to X$ be a map. Fix
$n\ge1$ and $\varepsilon>0$. Consider another metric $\rho^{(n)}$
which takes into account the distance between the respective $n$
initial iterates of points, namely put $\rho^{(n)}(x,y)=\max_{0\le
j<n}\rho(f^j(x),f^j(y))$. A subset $E$ of $X$ is called
\emph{$(n,f,\varepsilon)$--separated} if for every two different
points $x,y \in E$ it holds $\rho^{(n)}(x,y)>\varepsilon$. We say
that a subset $F \subset X$ {\em $(n,f,\varepsilon)$--spans} $X$,
if for every $x \in X$ there is $y \in F$ for which
$\rho^{(n)}(x,y)\le\varepsilon$.

We by $sep\,(n,f,\varepsilon)$ denote the maximal possible
cardinality of an $(n,f,\varepsilon)$-separated set in $X$, and by
$span\,(n,f,\varepsilon)$ the minimal possible cardinality of a
set which $(n,f,\varepsilon)$-spans $X$.

Then the topological entropy of $f$ is defined by
\begin{eqnarray*}
h(f) = \lim_{\varepsilon\to 0} \limsup_{n\to \infty}\frac{1}{n}\,
\log sep \,(n,f,\varepsilon) = \lim_{\varepsilon \to
0}\limsup_{n\to \infty}\frac{1}{n}\, \log span\,(n,f,\varepsilon)
\end{eqnarray*}

The following well-known lemma (see for example \cite{ALM}) shows
a way of computation of entropy when a system can be divided into
the smaller subsystems.

\begin{lemma}\label{Lm:ALM}
If $X=\bigcup_{\alpha\in A}X_{\alpha}$ where each $X_{\alpha}$ is
closed and invariant set then $h(f)=\sup_{\alpha\in
A}h(f|_{X_{\alpha}})$.
\end{lemma}

Recall that $\Con(X)$ denotes the space of all subcontinua of $X$
endowed with the Hausdorff metric. Given a dynamical system
$(X,f)$, by its \emph{connected envelope} we mean the system
$(\Con(X),\mathcal{F})$, where $\mathcal{F}:\Con(X)\to\Con(X)$ is
given by $\mathcal{F}(A)=f(A)$, where, as usual, $f(A)$ denotes
the set of all $f(x)$, $x\in A$. Clearly, the system
$(\Con(X),\mathcal{F})$ contains a copy of the original system
$(X,f)$ (consider the subspace of all singletons $\{x\}$, $x\in
X$). In \cite{Matv1} it was proved that topological entropy of an
interval dynamical system is equal to that of its connected
envelope. In \cite{KwOp} the same was proved for transitive
systems on graphs. Our next theorem establishes this equality for
any dynamical system on a tree.

\begin{theorem}\label{Tm:EntrCon}
Let $(T,f)$ be dynamical system on a tree and
$(\Con(T),\mathcal{F})$ be its connected envelope. Then
$h(\mathcal{F})=h(f)$.
\end{theorem}

\begin{proof}
The proof is based on Theorem~\ref{Tm:Dicho} and
Lemma~\ref{Lm:ALM}. Consider the family of closed invariant sets
$\{N_A\}_{A\in \Con(T)}$ where $N_A=$
$\overline{\{f^n(A):n\ge0\}}$. Since each $A\in N_A$, $\Con(T)$ is
the union of all $N_A$. Now, we can apply Lemma~\ref{Lm:ALM}:
$$h({\mathcal F})=\sup_{A\in \Con(T)}h({\mathcal
F}|_{N_A}).$$

Let $A\in \Con(T)$ is given. If $A$ is asymptotically periodic,
then it can be derived directly from the definition of the
topological entropy that $h({\mathcal F}|_{N_A})=0$. Otherwise, by
Theorem~\ref{Tm:Dicho}, $A$ is asymptotically degenerate. So, the
$\omega$-limit set $\omega_{{\mathcal F}}(A)$ is a subset of
$T_{sing}:=\{\{x\}:x\in T\}$. Thus $h({\mathcal
F}|_{N_A})=h(f|_{\omega_{{\mathcal F}}(A)})\le h({\mathcal
F}|_{T_{sing}})=h(f)$.

We see that $h({\mathcal F}|_{N_A})\le h(f)$ for every $A\in
\Con(T)$. In view of Lemma~\ref{Lm:ALM} this implies inequality
$h({\mathcal F})\le h(f)$. The converse inequality holds, because
$(T,f)$ is a subsystem of $(\Con(T),{\mathcal F})$.
\end{proof}

Recall that the \emph{Hausdorff distance} between two sets $A_1$
and $A_2$ in a metric space $X$ is given by
$d_H(A_1,A_2)=\inf\{\varepsilon>0 : A_1 \subseteq
\overline{B}(A_2,\varepsilon)\,\textnormal{and}\,A_2 \subseteq
\overline{B}(A_1,\varepsilon)\}$ where
$\overline{B}(A,\varepsilon)$ denotes the union of all closed
balls of radius $\varepsilon > 0$ whose centres run over $A$. This
is a metric on the family of all bounded, nonempty closed subsets
of $X$. As we remarked above, the Hausdorff metric generates the
same topology on $\Con(X)$ as that given by $\liminf$ and
$\limsup$.

Recall the definition of a functional envelope of a dynamical
system  (see \cite{AKoS}). For the general references see
\cite{Kol,Rom,ShMR,ShR}. Given a metric space $(X,\rho)$, denote
the set of all continuous maps $X\to X$ by $S(X)$. We endow the
space $S(X)$ with the Hausdorff metric $\rho_H$ (derived from the
metric $\rho_{max}((x_1,y_1),(x_2,y_2))=\max\{\rho(x_1,x_2),$
$\rho(y_1,y_2)\}$ in $X\times X$) applied to the graphs of maps.
Denote the corresponding metric space by $S_H(X)$. Given a
dynamical system $(X,f)$, consider the uniformly continuous map
$F:S_H(X)\to S_H(X)$ defined by $F(\varphi)=f\circ\varphi$ (first
apply $\varphi$) for any $\varphi\in S_H(X)$. The space $S_H(X)$
is \emph{not} compact (because it is not complete). However, if we
view $S_H(X)$ as a subset of the space of all closed subsets of
$X\times X$ endowed with the Hausdorff metric, then the closure
$\overline{S_H(X)}$ will be compact. The uniformly continuous map
$F$ can be uniquely extended to a continuous selfmap of a compact
metric space $\overline{S_H(X)}$. We will denote this map by the
same letter $F$ as well; that is
$F:\overline{S_H(X)}\to\overline{S_H(X)}$. The system
$(\overline{S_H(X)},F)$ is called a \emph{functional envelope} of
$(X,f)$. Again, as in the case of connected envelope, the system
$(\overline{S_H(X)},F)$ contains a copy of the original system
$(X,f)$ (consider the subspace of all constant maps).

If $X=T$ is a tree, then the extension $F:\overline{S_H(T)}\to
\overline{S_H(T)}$ can be described precisely in the following
way. Recall that a set-valued map $M:T\to T$ is \emph{upper
semicontinuous} if for every point $x\in T$ and every open subset
$V$ of $T$ such that $V\supseteq M(x)$ the set $\{y\in
T:M(y)\subseteq V\}$ contains a neighbourhood of $x$. One can
prove that $\overline{S_H(T)}$ consists of graphs of all
set-valued maps $T\to T$ which have nonempty connected, compact
values and are upper semicontinuous, and the extension
$F:\overline{S_H(T)}\to \overline{S_H(T)}$ is given by
$F(\varphi)=\mathcal{F}\circ\varphi$ for any $\varphi\in
\overline{S_H(T)}$.

In \cite{Matv1} it was proved that if an interval dynamical system
has zero topological entropy, then so does its functional
envelope. Now, we are going to prove the generalization of this
result for dynamical systems on trees. To do this, we need the
following estimates on the numbers used in the definitions of
topological entropy.

\begin{lemma}\label{Lm:Ineq}
Let $(T,f)$ be dynamical system on a tree, $(\Con(T),\mathcal{F})$
be its connected envelope and $(\overline{S_H(T)},F)$ be its
functional envelope. Then for any $\varepsilon>0$, $n\ge 1$ it
holds
$$
sep\left(n,f,\varepsilon\right)^{N_1(\varepsilon)}\le
sep\left(n,F,\varepsilon\right)\le
span\left(n,\mathcal{F},\varepsilon/2\right)^{N_2(\varepsilon)},
$$
for some numbers $N_{1,2}(\varepsilon)$ which do not depend on $n$
and $N_{1,2}(\varepsilon)\to+\infty, \varepsilon\to 0+$.
\end{lemma}

\begin{proof}
Fix $\varepsilon>0$ and $n\ge 1$. First, let us prove the
right-hand inequality. Let $\{T_k\}_{k=1}^N$ be a cover of $T$
with continua of diameter less than $\varepsilon$. Then for each
pair $\varphi,\psi\in \overline{S_H(T)}$ the inequality
$\rho_H(\varphi,\psi)>\varepsilon$ implies
$d_H(\varphi(T_k),\psi(T_k))>\varepsilon$ for some $k$ (here $d_H$
denotes the Hausdorff metric on the space $\Con(T)$ and $\rho_H$
denotes the Hausdorff metric on the space $\overline{S_H(T)}$).
Moreover, $\rho_H^{(n)}(\varphi,\psi)>\varepsilon$ implies
$d_H^{(n)}(\varphi(T_k),\psi(T_k))>\varepsilon$ for some $k$. Now,
suppose that there is an $(n,F,\varepsilon)$-separated set $E_0$
of cardinality $M^N+1$ where $M$ is minimal possible cardinality
of a set in $\Con(T)$ which $(n,\mathcal{F},\varepsilon/2)$-spans
$\Con(T)$. Consecutively, for each $1\le k\le N$, by Dirichlet's
box principle, we take a subset $E_k\subset E_{k-1}$ of
cardinality $M^{N-k}+1$ such that
$d_H^{(n)}(\varphi(T_k),\psi(T_k))\le\varepsilon$. On the last
step we get a set $E_N\subset E_0$ which contains two different
elements $\varphi,\psi$ such that
$d_H^{(n)}(\varphi(T_k),\psi(T_k))\le\varepsilon$ for each $1\le k
\le N$. This implies $\rho_H^{(n)}(\varphi,\psi)\le\varepsilon$, a
contradiction to the fact that $E_0$ is
$(n,F,\varepsilon)$-separated set. Thus the maximal possible
cardinality of an $(n,F,\varepsilon)$-separated set is less than
or equal to $M^N$. We put $N_2(\varepsilon)=N$.

Now, we are going to prove the left-hand inequality. Let
$I\subseteq T$ be an edge in $T$. For convenience, we assume that
$I=[0,1]$. Let $x_k=\frac{k}{K}, 0\le k\le K$ where
$K=[\frac{1}{2\varepsilon}]-1$. (It suffices to prove the
inequality for small enough $\varepsilon$, so we can assume that
$K\ge 1$.) Let $F$ be an $(n,f,\varepsilon)$-separated set in $T$
of the maximal possible cardinality. For any $K$-tuple
$\overline{y}=(y_1,y_2,\dots,y_K)$ of elements of $F$ we define
the (multivalued) map $\varphi_{\overline{y}}\in
\overline{S_H(T)}$ by
\begin{itemize}
    \item $\varphi_{\overline{y}}(x) = \{y_j\}$, if
$x\in(x_{j-1},x_{j})\subset I$, for some $1\le j\le K$,
    \item $\varphi_{\overline{y}}(x)
=T$, if $x=x_{j}$ for some $0\le j\le K$, or $x\in T\setminus I$.
\end{itemize}
One can see that collection
$\{\varphi_{\overline{y}}\}_{\overline{y}\in F^K}$ forms an
$(n,F,\varepsilon)$-separated set in $\overline{S_H(T)}$. Thus
$sep\left(n,F,\varepsilon\right)\ge |F|^K=
sep\left(n,f,\varepsilon\right)^K$. We put ${N_1(\varepsilon)}=K$.
\end{proof}

\begin{theorem}\label{Tm:EntrFun}
Let $(T,f)$ be dynamical system on a tree and $(S_H(T),F)$ be its
functional envelope.
\begin{enumerate}
    \item If $h(f)=0$, then $h(F)=0$.
    \item If $h(f)>0$, then $h(F)=+\infty$.
\end{enumerate}
\end{theorem}

\begin{proof}
Let $h(f)=0$. Let $(\Con(T),\mathcal{F})$ be connected envelope of
$(T,f)$. Then, by Theorem~\ref{Tm:EntrCon}, $h(\mathcal{F})=0$. By
right-hand inequality in Lemma~\ref{Lm:Ineq} we get
$$
\limsup_{n\to\infty} \frac{1}{n} \log
sep\left(n,F,\varepsilon\right)\le
N_2(\varepsilon)\limsup_{n\to\infty} \frac{1}{n} \log
span\left(n,\mathcal{F},\varepsilon/2\right),
$$
for every $\varepsilon>0$. Since $h(\mathcal{F})=0$, the
right-hand side of the last inequality equals $0$ for any
$\varepsilon>0$. So, $h(F)=0$.

Let $h(f)>0$. Then, by left-hand inequality in
Lemma~\ref{Lm:Ineq}, we get
$$
\limsup_{n\to\infty} \frac{1}{n} \log
sep\left(n,F,\varepsilon\right)\ge
N_1(\varepsilon)\limsup_{n\to\infty} \frac{1}{n} \log
sep\left(n,f,\varepsilon\right),
$$
for every $\varepsilon>0$. Since $N_1(\varepsilon)\to+\infty,
\varepsilon\to 0+$, we see that $h(F)\ge Ch(f)$ for any positive
$C$. So, $h(F)=+\infty$.
\end{proof}

\textbf{Acknowledgements.}
The paper was essentially written during the participation of the
author in CODY Autumn in Warsaw '10. The kind hospitality of the
Institute of Mathematics of PAN and the Banach Centre is highly
appreciated.

\bibliographystyle{gtart}

\end{document}